\documentclass[11pt,a4paper, twoside]{amsart}
\usepackage[utf8]{inputenc}

\usepackage{tikz}

\usepackage{amsmath}
\usepackage{amsfonts}
\usepackage{amssymb}
\usepackage{latexsym}

\usepackage{hyperref}
\hypersetup{backref, pdfpagemode=FullScreen, colorlinks=true,
  citecolor=magenta, linkcolor=cyan, urlcolor=blue}

\usepackage{mathtools}
\mathtoolsset{showonlyrefs}

\usepackage{graphicx}
\usepackage{xcolor}
\usepackage{graphicx} 

\usepackage{version}
\usepackage{todonotes}

\newtheorem{theo}{Theorem}[section]
\newtheorem*{theo*}{Theorem}
\newtheorem{lemma}[theo]{Lemma}

\DeclareMathOperator{\vol}{vol}

\newcommand{\ip}[2]{\left\langle #1,#2\right\rangle}
\newcommand{\bnorm}[1]{\left\vert#1\right\vert}

\newcommand{\Kn}{{\mathcal K}^n} 
\newcommand{\Ksn}{{\mathcal K}_{e}^n} 
\newcommand{\Knull}{{\mathcal K}_{(o)}^n} 
\newcommand{\Kcn}{{\mathcal K}_c^n} 

\newcommand{\R}{\mathbb{R}} 
\newcommand{\Z}{\mathbb{Z}} 

\newcommand{\conevol}[1]{\operatorname{V}_{#1}} 
\newcommand{\regbd}{\operatorname{bd}^{'}} 

\newcommand{\pyr}[2]{[#1,#2]} 
\newcommand{\C}[3]{C_{#1,#2}^{(#3)}}
\newcommand{\F}[3]{F_{#1,#2}^{(#3)}}

\newcommand{\conv}{\mathrm{conv}\,} 
\newcommand{\suk}{\mathrm{h}}
\newcommand{\lin}{\mathrm{lin}\,}

\newcommand{\bd}{\mathrm{bd}\,}
\DeclareMathOperator{\dint}{d}
\newcommand{\sph}{\mathbb{S}}

\newcommand{\sura}{\mathrm{S}}

\newcommand{\haus}{\hausdorff}

\newcommand{\V}{\mathrm{V}}
\DeclareMathOperator{\cen}{c}
\DeclareMathOperator{\EE}{\mathbb{E}}
\DeclareMathOperator{\PP}{\mathbb{P}}
\newcommand{\hausdorff}{\mathcal{H}}
\DeclareMathOperator{\aff}{aff}

\numberwithin{equation}{section}

\begin{document}
	
\title[Affine subspace inequality]{The affine subspace concentration inequality for centered convex bodies}
\author{Katharina Eller}
\address{Institut f\"ur Mathematik,
Technische Universit\"at Berlin,
Stra\ss e des 17.\ Juni 136,
10623 Berlin, Deutschland}
\email{eller@math.tu-berlin.de}
\author{Ansgar Freyer}
\address{Institut f\"ur Diskrete Mathematik und Geometrie,
Technische Universit\"at Wien,
Wiedner Hauptstra\ss e 8-10/1046,
1040 Wien, Austria}
\email{ansgar.freyer@tuwien.ac.at}

	
\begin{abstract}
An affine version of the linear subspace concentration inequality as proposed by Wu in \cite{wu2023affine} is established for centered convex bodies. This generalizes results from \cite{wu2023affine} and  \cite{FHK} on polytopes to convex bodies.
\end{abstract}
\maketitle

\section{Introduction}
Let $\Kn$ denote the set of convex bodies in $\R^n$,
i.e., the family of all convex and compact subsets
$K\subset\R^n$ with non-empty interior. The subfamily of convex bodies 
containing the origin in their interior is denoted by $\Knull$ and the subset of origin-symmetric convex bodies, i.e., the sets $K \in\Kn$ satisfying $K=-K$, is denoted by $\Ksn$. 
A convex body $K\in\Kn$ is called centered if its
centroid $\cen(K)$ is located at the origin, i.e.,
\begin{equation*}
  \cen(K)= \frac{1}{\vol_n(K)}\int_K x \, \dint \hausdorff^{n}(x)=0,
\end{equation*}
where $\haus^n$ denotes 
 the  $n$-dimensional Hausdorff measure. We denote by $\vol_n$ the $n$-dimensional volume functional, i.e., the restriction of $\haus^n$ to the set of $n$-dimensional convex bodies. The set of all  centered convex bodies in $\R^n$ is denoted by
$\Kcn$.
For $x,y \in \R^n$ let $\ip{x}{y}$ denote the
standard inner product on $\R^n$, and $\bnorm x=\sqrt{\ip{x}{x}}$ the Euclidean norm of $x$. 

The study of geometric measures related to convex bodies is a corner stone of convex geometry. One of the central problems in the classical Brunn-Minkowski theory is the
{Minkowski-Christoffel problem} asking for necessary and sufficient conditions
characterizing the surface area measures of a convex body among the finite
Borel measures on the sphere. For a definition of the surface area measures
and an overview of the Minkowski-Christoffel problem we refer to
\cite[Chapter 8]{schneider_2013}. In addition to the surface area measures, another important geometric measure is the {cone volume measure} $\conevol{K}$ of a convex body $K\in\Knull$.
It is defined for a Borel set $\omega \subseteq \sph^{n-1}$ of the sphere by 
\begin{align*}
   \conevol{K}(\omega)= \frac{1}{n} \int_{\nu_K^{-1}(\omega)} \langle x , \nu_K(x)\rangle \dint\hausdorff^{n-1}(x),
\end{align*}
where $\nu_K$ is the Gauss map of $K$, i.e., $\nu_K$ is defined on the regular boundary points $\regbd(K)$ of $K$ and assigns each such boundary point its unique outer unit normal vector.
The interest in the cone volume measure also stems from the fact that it is a particular instance of an $L_p$ surface area measure (see \cite[Section 9.2]{schneider_2013}) that allows for a geometric interpretation. The geometry behind the cone volume measure is best explained for polytopes: Let $P = \{x\in\R^n \colon \langle a_i,x\rangle \leq 1,~1\leq i\leq m\}$ be a polytope in $\Knull$. We assume that the inequalities $\langle a_i,x\rangle\leq 1$ are {irredundant}, i.e., each of them defines a proper facet $F_i$ of $P$. If we denote by $C_i\subset P$ the convex hull of the facet $F_i$ with the origin (a ``cone'' in $P$), one can verify that
\begin{equation}
    \label{eq:cone_vol_polytope}
    \V_P(\omega) =  \sum_{i:~\tfrac{a_i}{|a_i|}\in \omega}\vol (C_i)
\end{equation}
holds for all $\omega\subset\sph^{n-1}$. As in the case of the surface area measures, it is natural to ask for necessary and sufficient conditions on a Borel measure $\mu$ to be the cone volume measure of a convex body $K\in\Knull$, which leads to the log-Minkowski problem. For \emph{even} measures, i.e., measures $\mu$ with $\mu(-\omega) = \mu(\omega)$, this question was settled by B\"or\"oczky, Lutwak, Yang and Zhang in \cite{boroczky2013logarithmic}. They showed that an even non-zero finite Borel measure $\mu$ on $\sph^{n-1}$ can be represented as $\mu=\conevol{K}$ for some $K\in\Ksn$, if and only if $\mu$ satisfies the (linear) {subspace concentration condition}, i.e., we have
  \begin{equation}
                 \mu(L \cap \sph^{n-1})\leq  \frac{\dim (L)}{n} \mu(\sph^{n-1}), 
   \label{eq:scc}
  \end{equation}
  for all proper linear subspaces $L\subset\R^n$, and whenever equality
  holds in \eqref{eq:scc} for some $L$ then there exists a complementary linear subspace
  $L'$ such that $\mu$ is concentrated on $(L\cup
  L')\cap \sph^{n-1}$. 


In the general case, the log-Minkowski problem is still open, although further results were obtained in recent years \cite{ boroczky2022log, BHZ, boroczky2016cone, boroczky2014conevolume, chen2019logarithmic, henk2014cone, zhu2014logarithmic}. For an overview on the state of the art regarding the cone volume measure and related problems we refer to the survey article \cite{Boroczky2022}.
In particular, it was shown by Henk and Linke in \cite{henk2014cone} that the cone volume measures of centered polytopes satisfy the subspace concentration condition. This was generalized to arbitrary centered convex bodies B\"or\"oczky  and Henk in \cite{boroczky2016cone}. We state the inequality here.
\begin{theo}[\cite{boroczky2016cone}] \label{thmmain: sci cone volume} Let $K \in \Kcn$, then
  \begin{equation}
  \label{eq: aff conc cond}
                 \conevol{K}(L \cap \sph^{n-1})\leq  \frac{\dim (L)}{n} \conevol{K}(\sph^{n-1}), 
  \end{equation}
  holds for all linear subspaces $L\subset\R^n$.
\label{thm:conesym}  
\end{theo}

Recently, Wu formulated an affine version of the subspace concentration condition for polytopes \cite{wu2023affine}: Again, let $P=\{x\in\R^n \colon \langle a_i,x\rangle \leq 1,~1\leq i\leq m\}$, where each vector $a_i$ defines a facet of $P$, and let $C_i$, $1\leq i\leq m$, be the corresponding cones in $P$. Then $P$ is said to satisfy the affine subspace concentration condition, if
 \begin{equation}
 \label{eq:affine_sc_poly}
      \sum_{i : a_i \in A} \vol (C_i) \leq \frac{\dim (A)+1}{n+1} \vol(P)
 \end{equation}
 holds for all proper affine subspaces $A\subset \R^n$ and, moreover, equality holds, if and only if there exists a complementary affine subspace $A'\subset \R^n$ with $\{a_1,\dots,a_m\}\subset A\cup A'$. Here, the term ``complementary'' is meant in the sense of affine spaces, i.e., $A\cap A' = \emptyset$ and $\aff(A\cup A')=\R^n$.

In \cite{wu2023affine} the affine subspace concentration condition is proven for polytopes that are in addition smooth and reflexive lattice polytopes. We shall not give the definition here, but refer to \cite{wu2023affine} instead. It seems that the extra condition of $P$ being a smooth reflexive lattice polytope might not be necessary. In \cite{FHK} it is shown that the inequality \eqref{eq:affine_sc_poly} holds for all centered polytopes and all affine subspaces. However, the equality cases remain open in general.

Note that the normalization of the inequalities $\langle a_i, x\rangle\leq 1$ determining $P$ to have right hand side 1 is crucial in the affine subspace concentration condition. Renormalizing the vectors $a_i$ would create new affine dependencies and thus a different type of affine subspace concentration condition. In our setting the normalization is natural in that the vectors $a_i$ are the vertices of the polar polytope $P^\star$. Since the cone volume measure $V_P$ is supported on the vectors $\tfrac{a_i}{|a_i|}$, we can restate \eqref{eq:affine_sc_poly} in terms of the cone volume measure as follows:
\begin{equation}
    \label{eq:affine_sc_poly2}
    \conevol{P}(\pi (\bd(P^\star) \cap A)) \leq \frac{\dim (A)+1}{n+1}\,\conevol{P}(\sph^{n-1}),
 \end{equation}
where $\pi\colon\R^n\setminus\{0\}\to\sph^{n-1}$ denotes the central projection to the sphere.
Presented this way, the inequality could be stated for any convex body in $K\in\Knull$, where $K^\star = \{a\in\R^n \colon \langle a,x\rangle \leq 1,~\forall x\in K\}$ is the polar convex body. We show that it indeed extends to centered convex bodies.

\begin{theo} \label{theo: maintheo}
    Let $K \in \Knull$ be centered and let $A \subset \R^n$ be an affine subspace. Then it holds
\begin{equation}
    \label{eq:affine_sc_conv2}
    \conevol{K}(\pi (\bd(K^\star) \cap A)) \leq \frac{\dim( A)+1}{n+1}\conevol{K}(\sph^{n-1}).
 \end{equation}
\end{theo}

In the proof we make use of a lifting argument from \cite{FHK} for polytopes that we adapt to convex bodies.
Moreover, we provide characterizations of the equality case in Theorem \ref{theo: maintheo} for two special cases in Section \ref{sec:equality}.

\section{Preliminaries and Notation}\label{sec:pre}
 For a comprehensive overview of the basic concepts and definitions in convex geometry we refer to \cite{schneider_2013}. We write
$\sph^{n-1}$ for the $(n-1)$-dimensional Euclidean unit sphere, i.e., 
$\sph^{n-1}=\{ x\in\R^n : \bnorm{x} = 1 \}$.
For a convex body $K$ let $\bd(K)$ denote its boundary points.

For a given convex body $K \in \Kn$ the \emph{support function} $\suk_K:
\mathbb{R}^n \to \mathbb{R}$
is defined by
\begin{equation*} 
\suk_K(u)=\max_{x \in K} \ip{u}{x}.
\end{equation*}
The support function is continuous and positively homogeneous of degree 1. The hyperplane
\(
  H_K(u)=\{x \in \mathbb{R}^n :   \ip{u}{x}=\suk_K(u)\}
\)
is a supporting hyperplane of $K$ and for a boundary
point $v \in \bd K\cap H_K(u)$, the vector $u$ will be called an
\emph{outer normal vector}. If in addition $u\in \sph^{n-1}$ then $u$ is an
\emph{outer unit normal vector}. Let $\regbd (K)\subseteq \bd(K)$ be the set of all boundary points
having a unique outer unit normal vector, it is called the regular boundary.
We remark that  the set of boundary points not having a unique outer normal vector has measure zero, that is $\hausdorff^{n-1}(\bd (K)
\setminus \regbd (K) )=0$.
The \emph{Gauss map} (or \emph{spherical image map})
$\nu_K: \regbd( K )\to \mathbb{S}^{n-1}$ maps a point $x$ to
its unique outer unit normal vector.

 The \emph{central projection} $\pi_k : \R^k \setminus \{0\} \to \sph^{k-1}$ is given by $\pi_k(x)=\vert x \vert ^{-1} x$. We simply write $\pi$ instead of $\pi_k$ if the ambient dimension is clear from the context.


The cone volume can also be expressed as (cf.\ Section 9.2 in \cite{schneider_2013})
\begin{align*}
    \conevol{K}(\omega)= \frac{1}{n} \int_{\omega} \suk_K(u) \dint \sura_K(u)
\end{align*}
where $ \sura_K(\omega) = \hausdorff^{n-1}(\nu^{-1}_K(\omega))$ denotes the surface area measure. We will also use a description of the cone volume measure in terms of the volume of sets that we call \emph{star pyramids}. A star pyramid $\pyr{Q}{v}$ with apex $v$ and (not necessarily convex) base $Q$ is the union of closed segments from $Q$ to $v$, i.e., \(
        \pyr{ Q}{v}= \bigcup_{x\in Q} [x,v].
    \)
If $Q$ is convex, we have $[Q,v]= \conv (Q \cup \{v\})$.
\begin{lemma} [Lemma 9.2.4 in \cite{schneider_2013}]\label{lem : Schneider}
Let $K \in \Knull$ and $\omega \subset \sph^{n-1}$ be a Borel set. Then,
    \begin{align*}
        \conevol{K}(\omega)= \hausdorff^n \left(\pyr{\nu_K^{-1}(\omega)}{0} \right).
    \end{align*}

\end{lemma}
\section{Affine subspace concentration for centered convex bodies} \label{sec : main}



The volume formula for pyramids extends to star pyramids as introduced in Section \ref{sec:pre}.
\begin{lemma} \label{lem : pyramid formula}
    Let $Q \subset \R^{n-1}$, $h \neq 0$. Then it holds
    \begin{align*}
        \hausdorff^n(\pyr{Q\times \{h\}}{0})=\frac{h}{n} \hausdorff^{n-1}(Q).
    \end{align*}
\end{lemma}
\begin{proof}
Note that it holds $\pyr{Q\times \{h\}}{0} \cap \{x \in \R^n  :  x_n=t\} = \frac{t}{h}Q\times\{t\}$. Then we obtain by Cavalieri's principle and the translation invariance of the Hausdorff measure
\begin{align*}
     \hausdorff^n(\pyr{Q\times \{h\}}{0})&= \int_0^h \hausdorff^{n-1}\left(\pyr{Q\times \{h\}}{0} \cap \{x \in \R^n  :  x_n=t\}\right) \dint t \\
    &= \int_0^h \left (\frac{t}{h} \right)^{n-1} \hausdorff^{n-1}( Q) \dint t= \frac{h}{n} \hausdorff^{n-1}(Q).\qedhere
\end{align*}
\end{proof}



 For a set $Q \subset \R^{k}$ we will denote
\(\widehat{Q} \coloneq Q \times \{1\} \subset \R^{k+1}\) or $(Q)^{\wedge}$ for longer arguments.
 For $K \in \Knull$ and $j\geq 1$ we define pyramids by \begin{equation} \label{eq : pyramid defi}
     K^{(j)}=\pyr{(K^{(j-1)})^{\wedge}}{-(n+j)e_{n+j}}\end{equation} where $K^{(0)}=K$. The pyramids $K^{(j)}$ are convex bodies in $\R^{n+j}$. 

    \begin{lemma} \label{lem : basic prop pyramid}
    Let $K \in \Knull$. Then the following holds:
    \begin{enumerate}
    \item If $K$ is centered, $K^{(1)}$ is centered.
       \item $\operatorname{vol}_{n+1}(K^{(1)})= \frac{n+2}{n+1}\operatorname{vol}_n(K)$.

       \item If $y\in \bd (K^\star)$ is an outer normal vector of $x\in K$, then $(\tfrac{n+2}{n+1}y,\, -\tfrac{1}{n+1})^\top\in \bd ((K^{(1)})^\star)$ is an outer normal vector of $K^{(1)}$ at each point in $[(x,1)^\top, -(n+1)e_{n+1}]$.

       \item A regular boundary point of $K^{(1)}$ is either in the relative interior of $\widehat K$, or of the form $\mu (x,1)^\top - (1-\mu)(n+1)e_{n+1}$ for some $x \in \regbd (K)$, $\mu\in (0,1)$.
    \end{enumerate}
\end{lemma}
\begin{proof}
The first statement is another application of Cavalieri's principle (see \cite[Lemma 2.1]{FHK} for details). 
The second statement follows directly from Lemma \ref{lem : pyramid formula}.

For the third statement let $\overline y = (\tfrac{n+2}{n+1}y,\, -\tfrac{1}{n+1})^\top$ and $\hat z = (z,1)^\top$, where $z\in K$ as well as $\mu \in [0,1]$. Then we have
\[
\langle \overline y, \mu\hat z - (1-\mu)(n+1)e_{n+1}\rangle = \mu\frac{n+2}{n+1}\langle y, z\rangle - \mu \frac{1}{n+1} + (1-\mu).
\]
Since $y\in \bd (K)$ we see that 
\[
\langle \overline y, \mu\hat z - (1-\mu)(n+1)e_{n+1}\rangle \leq 1
\]
with equality for $z=x$, which proves the claim. The last statement is a direct consequence of iii).
\end{proof}

   To keep track of the part of the boundary corresponding to an affine subspace $A \subset \R^n$ within the sequence of pyramids $K^{(j)}$ we introduce the following notation for $j \geq 1$
    \begin{equation}\label{eq:F}
        \F{K}{A}{j}= \pyr{\, (\F{K}{A}{j-1})^{\wedge}} {-(n+j)e_{n+j}} \subset K^{(j)}
         \end{equation}
    where \begin{align*}
        \F{K}{A}{0}= \nu_K^{-1}(\pi_n(\bd (K^\star) \cap A)).
    \end{align*}

    The corresponding ``cones'' are star pyramids defined by 
    \begin{align*}
         \C{K}{A}{j}&= \pyr{\F{K}{A}{j}}{0} 
    \end{align*}
    for $j \geq 0$.
\begin{lemma} \label{lem : cone pyramid}
    Let $K \in \Knull$ and $A \subset \R^n$ be an affine subspace. Then for $j \geq 1$,
    \begin{align*}
        \hausdorff^{n+j}(\C{K}{A}{j})=\hausdorff^{n}(\C{K}{A}{0}).
    \end{align*}
\end{lemma}
\begin{proof}
    First we show $ \hausdorff^{n+1}(\C{K}{A}{1})=\hausdorff^{n}(\C{K}{A}{0})$. Set $G_1= \pyr{(\C{K}{A}{0})^{\wedge}}{-(n+1)e_{n+1}}$ and $G_2=\pyr{(\C{K}{A}{0})^{\wedge}}{0}$.
    We claim that \begin{align*}
        \hausdorff^{n+1}(\C{K}{A}{1})=\hausdorff^{n+1}(G_1)-\hausdorff^{n+1}(G_2).
    \end{align*}
    Because $ \pyr{(\F{K}{A}{0})^{\wedge}}{0}$ has $\hausdorff^{n+1}$-measure zero - as it is part of the boundary of the convex body $\pyr{\widehat{K}}{0}$ - it suffices to show $$\C{K}{A}{1}=(G_1\setminus G_2) \cup \pyr{(\F{K}{A}{0})^{\wedge}}{0}.$$  Let $p \in \C{K}{A}{1}$, then by construction there exists $x \in (\F{K}{A}{0})^{\wedge}$ and $\lambda,\beta \in [0,1]$ with $p= \beta (\lambda x +(1-\lambda)(-(n+1))e_{n+1}$. Let $E= \lin (\{x,-(n+1)e_{n+1}\})$, this is a two dimensional subspace with $p \in E$. In fact $p$ is contained in the triangle $\operatorname{conv}(\{0,-(n+1)e_{n+1},x\})$. So if $p \notin [0,x]$ then it holds $p \notin \operatorname{conv}(\{0, e_{n+1},x\} )$. We obtain 
    \begin{align*}
        p \in &\operatorname{conv}(\{0,-(n+1)e_{n+1},x\}) \setminus \operatorname{conv}(\{0,e_{n+1},x\} ) \subseteq G_1 \setminus G_2.
    \end{align*}
    On the other hand let $p \in  G_1 \setminus G_2$. Then there exists some  $x \in (\F{K}{A}{0})^{\wedge}$  such that $p \in \operatorname{conv}(\{-(n+1)e_{n+1},e_{n+1},x\})$ and also $p \notin \operatorname{conv}(\{0,e_{n+1},x\})$. In particular, there exist $\lambda,\mu \in [0,1]$ with
    \begin{align*}
        p&=\lambda x+(1-\lambda)(\mu (-(n+1)e_{n+1})+(1-\mu)\cdot0)\\
        &=(\lambda +(1-\lambda)\mu)\left(\frac{\lambda}{\lambda +(1-\lambda)\mu}x+\frac{(1-\lambda)\mu}{\lambda +(1-\lambda)\mu}(-(n+1)e_{n+1})\right).
    \end{align*}
    This implies $p \in \C{K}{A}{1}$. Furthermore by construction we have $(\F{K}{A}{0})^{\wedge}\subseteq \F{K}{A}{1}$ and so $\pyr{(\F{K}{A}{0})^{\wedge}}{0}\subseteq \C{K}{A}{1}$. This yields the claim.
    
    Using Lemma \ref{lem : basic prop pyramid} we obtain
    \begin{align*}
        \hausdorff^{n+1}(\C{K}{A}{1})&=\hausdorff^{n+1}(G_1)-\hausdorff^{n+1}(G_2)\\&= \frac{n+2}{n+1}\hausdorff^{n}((\C{K}{A}{0})^{\wedge})- \frac{1}{n+1}\hausdorff^{n}((\C{K}{A}{0})^{\wedge})=\hausdorff^{n}(\C{K}{A}{0}).
    \end{align*}
     Analogously one can prove $ \hausdorff^{n+j}(\C{K}{A}{j})=\hausdorff^{n+j-1}(\C{K}{A}{j-1})$ for $j \geq 1$. This yields the desired statement.
\end{proof}

\begin{proof}[Proof of Theorem \ref{thmmain: sci cone volume}]
   
   For any $k\in\Z_{\geq 1}$, we consider the embedding  \begin{align*}\varphi_k: \R^k \to \R^{k+1},\quad
       \varphi_k(x)= \begin{pmatrix} \frac{k+2}{k+1}x \\ - \frac{1}{k+1}\end{pmatrix}
   \end{align*}
    and we write $\psi_{n+j} \coloneq \varphi_{n+j-1} \circ\cdots \circ \varphi_n$, where $j\geq 1$. For $j=0$ we define $\psi_n$ to be the identity on $\R^n$. We seek to apply the linear subspace concentration inequality to the convex bodies $K^{(j)}$ (introduced in \eqref{eq : pyramid defi}) and the linear spaces $L^{(j)}\coloneq \lin (\psi_{n+j}(A))\subseteq\R^{n+j}$.
   Note that we have $\dim (L^{(j)}) = \dim (A )+1$ for all $j >0 $. 

    Here we write $F^{(j)} \coloneq \F{K}{A}{j}$ for the boundary part of the $j$-th pyramid associated to $A$ (cf.\ \eqref{eq:F}). By Lemmas \ref{lem : Schneider} and \ref{lem : cone pyramid} we have
    \begin{equation}
    \label{eq:lift}
        \conevol{K} ( \pi_n ( \bd (K ^\star)\cap A ) ) = \hausdorff^n ( [F^{(0)},0] ) =  \hausdorff^{n+j}( [F^{(j)}, 0]).
    \end{equation}
    It follows from Lemma \ref{lem : basic prop pyramid} iii) that
    \begin{equation}
    \label{eq:claim_pt2}
        \psi_{n+j} (\bd K^\star ) \subseteq \bd((K^{(j)})^\star)
    \end{equation}
    holds for all $j\geq 0.$
    We claim that
    \begin{equation}
    \label{eq:claim_pt1}
        F^{(j)}\cap\regbd (K^{(j)}) \subseteq \nu_{K^{(j)}}^{-1} (\pi_{n+j}\circ \psi_{n+j} (\bd (K^\star) \cap A)).
    \end{equation}
    For $j=0$, this trivially follows from the definition of $F^{(0)}$. For $j>0$ we argue by induction. Let $x\in F^{(j)}\cap\regbd (K^{(j)})$. Then $x$ can be expressed as \[
    x=\mu \hat{y} + (1-\mu)(-(n+j)e_{n+j}),
    \]
    where $\mu\in (0,1)$, $y\in F^{(j-1)}\cap \regbd (K^{(j-1)})$ and $\hat{y} = (y,1)^\top\in \R^{n+j} $ (cf.\ Lemma \ref{lem : basic prop pyramid}). By induction, there exists an outer normal vector of $K^{(j-1)}$ at $y$ of the form $\psi_{n+j-1}(a)$, where $a\in \bd (K^\star) \cap A$. By \eqref{eq:claim_pt2}, we know that $\psi_{n+j-1}(a)\in \bd ((K^{(j-1)})^\star)$.
    Hence, Lemma \ref{lem : basic prop pyramid} iii) shows that $\psi_{n+j}(a)$ is the unique (up to positive multiples) outer normal vector of $K^{(j)}$ at $x$
    and \eqref{eq:claim_pt1} is proven.

    It follows from \eqref{eq:claim_pt1} and the fact that $L^{(j)}$ is a linear space that
    \[
    F^{(j)}\cap\regbd (K^{(j)}) \subseteq \nu_{K^{(j)}}^{-1} ( L^{(j)} \cap \sph^{n+j-1}).
    \]
    Since the Hausdorff measure of the non-regular boundary points of $K^{(j)}$ is zero, we obtain from \eqref{eq:lift} that
    \[
        \conevol{K}(\pi_n(\bd(K^\star)\cap A)) \leq \conevol{K^{(j)}} ( L^{(j)} \cap \sph^{n+j-1}).
    \]
    Since $K^{(j)}$ is centered and $L^{(j)}$ is $(\dim (A )+ 1)$-dimensional, the linear subspace concentration inequality (cf.\ Theorem \ref{thmmain: sci cone volume}) yields
    \[\begin{split}
        \conevol{K}(\pi_n(\bd(K^\star)\cap A)) &\leq \frac{\dim( A )+1}{n+j} \vol_{n+j}(K^{(j)})\\
        &= \frac{\dim (A) +1}{n+j}\cdot\frac{n+j+1}{n+1} \vol_{n+j}(K),
    \end{split}\]
    where we used Lemma \ref{lem : basic prop pyramid} to obtain the last line. Taking the limit $j\to\infty$ now yields the desired inequality.
    \end{proof}

\section{Equality cases}
\label{sec:equality}

Here we extend the characterization of the equality case in Theorem \ref{theo: maintheo} to the convex body setting. The corresponding results for centered polytopes are shown in \cite{FHK}.
\begin{theo} \label{theo : equality} 
    Let $K \in \Kcn$.
\begin{enumerate} \item If $A = \{a\}$, then equality holds in \eqref{eq: aff conc cond} if and only if K is a pyramid with base $Q=H_K(a)\cap K$.
\item If $A$ is a supporting hyperplane of $K^{\star}$ spanned by boundary point of $K^{\star}$, then equality holds in \eqref{eq: aff conc cond} if and
only if $K$ is a pyramid with apex $v_A$, where $v_A$ is the unique point in $K$ with $\langle v,a\rangle = 1$ for all $a\in A$.
\end{enumerate}
\end{theo}

Note that the point $v_A$ from ii) indeed exists. Since $A$ does not pass through the origin, it follows that $A$ can be uniquely expressed as $A = \{ a\in \R^n \colon \langle a,v_A\rangle = 1\}$ for some $v_A\in\R^n$. Since $A$ is supporting to $K^\star$ it follows that $v_A\in K$. 

\begin{proof}
 The proof of i) can be obtained along the same lines as in \cite[Theorem 1.2 i)]{FHK}; there, only the concavity properties of the section function w.r.t.\ $a^\perp$ are used.  

For ii), we adapt an argument from \cite{FHK}: For a set $\omega\subseteq\sph^{n-1}$ we consider the function
\[
f_\omega \colon K\to\R,\quad f_\omega(x) = \conevol{K-x}(\omega).
\]
This is an affine function on $K$. Indeed,
\begin{equation} \label{eq:fomega}
\begin{split}
\conevol{K-x}(\omega) &= \frac 1n \int_\omega \suk_{K-x}(u) \dint \sura(K-x,u)\\
&= \frac 1n \int_\omega \suk_K(u) -\langle x,u\rangle \dint\sura_K(u)\\ 
&= \conevol{K}(\omega) - \left\langle x,\int_\omega u\dint\sura_K(u)\right\rangle.
\end{split}
\end{equation}
In the following we consider the set $\omega = \pi(\bd (K^\star)\cap A ) $ and abbreviate $f=f_\omega$. Let $X$ be a uniformly distributed random vector in $K$. For convenience, we assume that $\vol_n(K)=1$ so that $f(x)\in [0,1]$ for all $x\in K$. Since $f$ is affine and $K$ is centered, we have
\(
\conevol{K}(\omega) = f(0) = \EE[f(X)].
\)
On the other hand, we have
\(
\EE[f(X)] = 1 - \int_0^1 \PP(f(X) < t)\dint t.
\)
 For $t\in [0,1]$, let $K(t) = \{x\in K\colon f(x)\leq t\}$ be the corresponding sublevel set of $f$. Note that since $f$ is affine, $K(t)$ is the intersection of $K$ with a halfspace whose normal vector $\int_\omega u \dint \sura_K(u)$ is the same for all $t$.

Note that the vector $v_A\in K$ satisfies (by definition) $\langle v_A, u\rangle = \suk_K(u)$ for all $u\in\omega$. Thus, we obtain from \eqref{eq:fomega} that
$v_A\in K(0)$. Let $m\in [0,1]$ denote the maximum of $f$. Then for any $t\in[0,1]$ we have the inclusion,
\[
K(t) \supseteq \frac tm K(m) + \frac{t-m}{m}K(0)\supseteq \frac tm K(m) + \frac{t-m}{m}v_A.
\]
Thus,
\begin{equation}
\label{eq:sublevel_est}
\PP(f(X) < t) = \vol_n ( K(t)) \geq \vol_n\left(\frac tm K(m)\right) = \left(\frac tm \right)^n.
\end{equation} 
Since we have $f(x) = 1$ for $x\in[m,1]$ we deduce
\[
\conevol{K}(\omega) = 1-\int_0^m \PP(f(X)<t) \dint t -( 1-m ) \leq \frac{mn}{n+1} \leq \frac{n}{n+1}
\]
This reproves the affine subspace concentration condition in this particular case. In order to have equality, it is necessary that $m=1$. Moreover, for the first inequality in the above to be tight, we need equality in \eqref{eq:sublevel_est}, i.e., $\vol_n (K(t)) = t^n$. Taking derivatives, it follows that $\hausdorff^{n-1}(\{x\in K \colon f(x) =t \}) = nt^{n-1}$. Since $f$ is an affine function this implies that $K$ is a pyramid with apex $v_A$ and base $\{x\in K\colon f(x) = m\}$.
\end{proof}

\section*{Acknowledgements}

Ansgar Freyer is partially supported by the Austrian Science Fund (FWF) Project
P34446-N.

\bibliographystyle{abbrv}
\bibliography{biblio}
\end{document}